\documentclass{amsart}

\usepackage{amssymb, amsmath}
\usepackage{amssymb,amsfonts,latexsym}
\usepackage{color}
\usepackage{graphicx}

\newtheorem{theorem}{Theorem}
\newtheorem{definition}{Definition}
\newtheorem{lemma}{Lemma}
\newtheorem{proposition}{Proposition}
\newtheorem{remark}{Remark}
\newtheorem{example}{Example}

\newcommand{\bee}[1]{\begin{equation}\label{#1}}
\newcommand{\beq}[1]{\begin{eqnarray}\label{#1}}
\newcommand{\ene}{\end{equation}}
\newcommand{\eqe}{\end{eqnarray}}

\begin{document}

\title[Irreducible characters and exponential growth]
{Irreducible characters of the symmetric group and exponential growth}

\author[Giambruno]{Antonio Giambruno}\address{Dipartimento di Matematica e
Informatica
\\Universit\`a di Palermo \\ Via Archirafi 34, 90123 Palermo, Italy}
\email{antonio.giambruno@unipa.it}

\author[Mishchenko]{Sergey Mishchenko}\address{Department of Algebra
and Geometric Computations \\
Ulyanovsk  State University \\ Ulyanovsk 432017, Russia}
\email{mishchenkosp@mail.ru}



\begin{abstract}
We consider sequences of degrees of ordinary irreducible $S_n$-characters.
We assume that the corresponding Young diagrams have
rows and columns bounded by some linear function of $n$
with leading coefficient less than one.
We show that any such sequence has  at least exponential growth and we compute
an explicit bound.
\end{abstract}
\maketitle

\section{Introduction}

This paper is devoted to the computation of a lower bound for the degree of some
irreducible characters of the symmetric group in characteristic zero.

Let $S_n$ be the symmetric group on $n$ symbols.
Let $\lambda\vdash n$ be a partition of $n$ and $f^\lambda$ the degree
of the irreducible $S_n$-character corresponding to $\lambda$.

We shall consider sequences of partitions $\{\lambda^{(n)}\}_{n\ge 1}$
where $\lambda^{(n)}\vdash n$ and the corresponding sequence of their degrees
$\{f^{\lambda^{(n)}}\}_{n\ge 1}$.
When the sequence of partitions is subject to suitable constraints,
there are several results in the literature about computing an upper bound of the sequence
$\{f^{\lambda^{(n)}}\}_{n\ge 1}$ (\cite{BR1}, \cite{gia-zaiBOOK}, \cite{OM}, \cite{mrz}, \cite{R1}).

For instance,  if each $\lambda^{(n)}= (\lambda^{(n)}_1, \lambda^{(n)}_2, \ldots)$
lies in a strip of fixed  height $k\ge 1$ (i.e., $\lambda^{(n)}_{k+1}=0$), then
$f^{\lambda^{(n)}} \le k^n$, for all $n\ge 1$.
In case the sequence has the further property that $\lambda^{(n)}_1\le \frac{n}{\alpha}$,
for some fixed $\alpha>1$, then an exponential lower bound can also be found.

The above example can be generalized to hooks.
In fact given $k,l\ge 0$, it can be shown that if we consider a sequence
of partitions lying inside the $k\times l$ hook (i.e., $\lambda^{(n)}_{k+1}\le l$),
then $f^{\lambda^{(n)}} \le (k+l)^n$ (\cite{R1}).
A lower bound in this case is found (see Proposition \ref{prop 3} below)
when the sequence has the further property that $\lambda^{(n)}_1, {\lambda^{(n)}_1}' \le \frac{n}{\alpha}$,
for some $\alpha>1$, where ${\lambda^{(n)}}'$ is the conjugate partition of $\lambda^{(n)}$.

Notice that in the above examples all partitions $\lambda^{(n)}$
have bounded diagonal $\delta(\lambda^{(n)})$.
But what can we say if $\delta(\lambda^{(n)})$ goes to infinity with $n$?
Here we shall consider such question.

In case $\lim_{n\to \infty} \frac{\delta({\lambda^{(n)}})^2}{n}= \varepsilon >0$,
then it is not hard to prove that the sequence $\{f^{\lambda^{(n)}}\}_{n\ge 1}$ has overexponential growth.
(see Proposition \ref{prop 2} below).

More generally here we shall prove the following result:
given any $\alpha\in \mathbb{R}$, $\alpha>1$, and
$\{\lambda^{(n)}\}_{n\ge 1}$ a sequence of partitions such that
$\lambda^{(n)}_1, {\lambda^{(n)}_1}' \le \frac{n}{\alpha}$,
then $f^{\lambda^{(n)}} \ge \beta^n$, for any $1< \beta< \alpha$,
for $n$ large enough. This result was proved by the second author
in \cite{M}, in case $\alpha$ is an integer.

A motivation for the study of such sequences of partitions
in found in the theory of algebras satisfying polynomial identities.
For instance as an application of this result (in the integral case)
in \cite{M}  it was proved that no variety of Lie algebras
can have exponential growth between one and two.
Also in \cite{GZ} in order to compute the exponential growth of a variety,
the authors introduced a real valued function $\Phi(\alpha_1, \ldots, \alpha_k)$
with the property that $\Phi(\alpha_1, \ldots, \alpha_k)^n$ asymptotically
equals $f^{\lambda}$, up to a polynomial factor,
where $\lambda=([\alpha_1n],\ldots, [\alpha_kn])$ (see also \cite{gmz1}, \cite{gmz2}).
This function is defined only for partitions lying in a strip of height $k$, while
in this paper no restriction on the height is considered.

\section{Preliminaries}

Throughout the paper we shall consider the ordinary representation theory of the symmetric group.
We refer to \cite{jk} for the notation and the basic notions.

If $\lambda=(\lambda_1, \ldots, \lambda_k)\vdash n$, we tacitely identify $\lambda$
with its Young diagram $D_\lambda = D_\lambda\{(i,j)\}$, where $(i,j)$ is the cell of $D_\lambda$
of coordinates $(i,j)$.
We also write $H_{ij}=\{ (k,l) \in \lambda\mid k\ge i, l=j \ \mbox{or} \
k=i, l\ge j\}$ for the hook of the cell $(i,j)$.
Recall that the hook number of $(i, j)$ is $h_{ij}=(\lambda_i-j)+(\lambda_j'-i)+1$ where
$\lambda_i-j$ is the length of the arm of the hook and
$\lambda_j'-i$ is the length of the leg.

We denote by $\lambda'=(\lambda_1', \ldots, \lambda_{\lambda_1}')$ the conjugate partition of $\lambda$,
where $\lambda_i'$ is the length of the $i$-th column of $\lambda$.

In what follows we shall use the hook formula (see \cite{jk}) and Stirling formula (see \cite{robbins}).

\begin{remark}\label{r1}
Let $\lambda\vdash n$ and let $T_\lambda=(t_{ij})$ be a standard tableau of shape $\lambda .$
If $N_{ij}=n+1-t_{ij},$ then $N_{ij}\ge h_{ij},$ for all $i,j.$
\end{remark}

\begin{proof}
Let $B_{ij} = \{ t_{kl} | t_{kl}\ge t_{ij} \}.$ Clearly $N_{ij}= |B_{ij}|.$
Since for each cell $(k,l)$ of the hook $H_{ij}$ of $(i,j)$ we have that
$t_{kl}\ge t_{ij},$ if follows that $\{ t_{kl} | (k, l)\in H_{ij} \} \subseteq B_{ij}.$
Thus  $N_{ij}\ge h_{ij}.$
\end{proof}

Recall that a corner cell of a diagram is a cell whose hook number is 1.

\begin{definition} If $\lambda \vdash n$ we define the diagonal of $\lambda$
as
$$
\delta(\lambda)= \max_i \{ \lambda_i\ge i \,\,\, and \,\,\, \lambda_i'\ge i \}.
$$
\end{definition}

We remark that $\delta=\delta(\lambda)$ is the size of the largest square diagram inside $\lambda .$

\begin{definition} If $k, l \ge 0$, we define the $k \times l$ hook as the set
$$
H(k, l)= \{ \lambda \vdash n | \lambda_{k+1}\le l \}.
$$
\end{definition}

\section{Two special cases}

\begin{proposition}\label{prop 3}
Let $\alpha\in \mathbb{R},$ $\alpha >1$ and $k\ge l\ge 0$.
If $\lambda \vdash n$ is such that
$\lambda\in H(k, l)$ and $\lambda_1, \lambda_1' \le \frac n \alpha$,
then
$$
f^{\lambda^{(n)}}\ge \frac {\alpha^n}{n^m},
$$
where $m=\frac{(2l+k-1)k}{2}.$
\end{proposition}

\begin{proof}
Let $t_1, \dots , t_{k+l}$ be defined as follows :
\begin{enumerate}
\item[1)]
if $1\le s \le l,$ \, set \  $t_s=\left\{
\begin{array}{ll}\lambda_s' & \text {if }\,\, \lambda_s' >0,
\\
0& \text {otherwise},
\end{array}
\right.
$
\item[2)]
\mbox{
if $1\le s \le k,$ \, set \
$t_{l+s}= \left\{
\begin{array}{ll} \lambda_s-l & \text {if }\,\, \lambda_s-l\ge 0,
\\
0& \text {if }\,\, \lambda_s-l < 0.
\end{array}
\right.
$
}
\end{enumerate}
We remark that $\sum_{i=1}^{k+l}t_i=n.$

Let $\mu=(l+k-1,\dots,l+1, l)\vdash m=\frac{(2l+k-1)k}{2}$ and define
$$
A=\{ (i, j) \mid (i,j)\in \lambda\cap\mu \},
$$
$$
B=\{  (i,j)\in \lambda\setminus\mu \mid i\ge k+1 \},
$$
$$
C=\{  (i,j)\in \lambda\setminus\mu \mid 1\le i \le k \}.
$$

Clearly since $|A|\le m,$
$$
\prod_{(i,j)\in A} h_{ij} \le n^m.
$$
The hook number of a cell in the $j$-th column of $\lambda$, $j\le l$, is $h_{k+i, j}< t_j-i+1,$
$1\le i \le \lambda_j'-k.$
Moreover consider the hook $H_{i, \mu_i+j}$ of a cell of $C$.
Its arm is $\lambda_i-\mu_i-j$ and its leg is $\le k-i$.
Hence
$$
h_{i, \mu_i+j}\le (\lambda_i-\mu_i-j) + (k-i) +1 =
t_{l+i}+l+k-\mu_i-j-i+1= t_{l+i}-j+1,
$$
since $\mu_i=l+k-i$.

 Hence
$$
\prod_{(i, j)\in B \cup C} h_{ij} < \prod_{i=1}^{k+l}t_i!.
$$

It follows that
$$
f^{\lambda}\ge \frac {1}{n^m}\,\, \frac {n!}{\prod_{i=1}^{k+l}t_i!}.
$$

By Stirling formula, recalling that $t_i\le \frac{n}{\alpha},$ we get
$$
\frac {n!}{\prod_{i=1}^{k+l}t_i!}\ge \frac {n^n\prod_{i=1}^{k+l}e^{t_i}}{e^n \prod_{i=1}^{k+l}t_i^{t_i}}\ge
 \frac {n^n}{\prod_{i=1}^{k+l}\frac{n^{t_i}}{\alpha^{t_i}}} =  \prod_{i=1}^{k+l}\alpha^{t_i}= \alpha^n.
$$

Hence we get that
$f^{\lambda}\ge \frac {\alpha^n}{n^m}$ and we are done.
\end{proof}

\begin{remark}\label{r2}
If $k< l$ and $\lambda\in H(k, l),$ then $\lambda'\in H(l, k)$ and $f^\lambda =f^{\lambda'}.$
Hence the conclusion of the previous proposition still holds when $k< l$ and $f^{\lambda}\ge \frac {\alpha^n}{n^m}$
where $m=\frac{(2k+l-1)l}{2}$ in this case.
\end{remark}

The following example illustrates the sets $A, B, C$ of the previous proposition.
\begin{example}
We consider the hook  $H(4,3)$ and a partition $\lambda \in H(4,3)$.
For $\lambda=(9,6,4,2,2,1)\vdash 24$, we have
\end{example}
\vskip.2cm

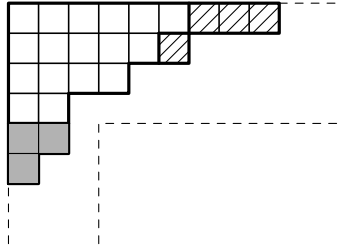
\centerline{
\includegraphics{alpha2014-pic1.mps}
}

\vskip.5cm
\noindent
{\em where $A$ consists of the white cells, $B$ of the grey cells and $C$
of the marked cells.
In this case we also have that $t_1=6, t_2=5, t_3=3, t_4=6, t_5=3, t_6=1, t_7=0$.}


\begin{lemma}\label{lemma 2}
Let $\{a_s\}_{s\ge 1},$ $\{b_s\}_{s\ge 1}$ be two sequences of natural numbers such that
$$
\lim_{s\rightarrow \infty }a_s=\lim_{s\rightarrow \infty }b_s = \infty
$$
and define the partition $\mu(s)=(b_s^{a_s}),$ $s\ge 1.$

Then for every $\beta>0$ there exists $s_0$ such that for every $s\ge s_0,$  $f^{\mu(s)}> \beta^{a_sb_s}.$
\end{lemma}

\begin{proof}
Let $n=ab$ and $\mu=\mu(s)=(b^a)\vdash n$.
Since $f^\mu=f^{\mu'},$ where $\mu'$ is the conjugate partition of $\mu$,
we assume, as we may, that $b\ge a$.

By the hook formula we get
$$
f^\mu=
\frac {n!}{(b!)^a}\cdot \frac {1!}{(b+1)} \cdot \frac {2!}{(b+1)(b+2)}\cdot \,
\cdots \,\cdot\frac {(a-1)!}{(b+1)\dots (b+a-1)}=
$$
$$
=\frac {n!}{(b!)^a}\cdot \prod_{t=1}^{a-1}{{b+t}\choose t}^{-1}>
\frac {n!}{(b!)^a}\cdot \prod_{t=1}^{a-1}\left(2^{b+t}\right)^{-1}=
\frac {n!}{(b!)^a}\cdot 2^{-n-\frac {a^2}{2}+\frac {2b+a}{2}}
 \frac {n!}{(b!)^a}2^{-2n}.
$$
since $a^2\le n.$ By Stirling formula, since $n=ab$ we get that
$f^\mu > \left( \frac a 4 \right)^n.$

Now, since $\lim_{s\rightarrow \infty }a_s = \infty ,$ for every
$\beta$ there exists $s_0$ such that $\frac {a_s}{4}\ge \beta,$ for any $s\ge s_0$
and the result follows.
\end{proof}

\begin{proposition}\label{prop 2}
Let $\alpha >1,$ $\alpha\in \mathbb{R},$ and let $\{\lambda^{(n)}\}_{n\ge 1}$ be a sequence
of partitions, $\lambda^{(n)} \vdash n.$ Suppose that for any $\varepsilon >0$,
$\frac{\delta(\lambda^{(n)})^2}{n}\ge \varepsilon >0$ holds for $n$ large enough.
Then for any $\gamma$ there exists $n_0$ such that $f^{\lambda^{(n)}}\ge \gamma^n$  for every $n \ge n_0.$
\end{proposition}

\begin{proof}
For every $n\ge 1$ set $\mu(n)=\left( \delta(\lambda^{(n)}\right)^{\delta(\lambda^{(n)})}\vdash k_n$
where $k_n = \delta(\lambda^{(n)})^2.$ Take $\beta=\gamma^{1/\varepsilon}.$
Then since $ \frac {k_n}{n}\ge \varepsilon$,
by Lemma \ref{lemma 2} there exists $n_0$ such that
$$
f^{\mu(n)}\ge \beta^{k_n}> \beta^{\varepsilon n}= \gamma^n,
$$
for any $n\ge n_0$.
Since $f^{\lambda^{(n)}}\ge f^{\mu(n)},$ the proof is complete.
\end{proof}

\section{The main results}

The following lemma which is of interest by itself,
 is the main tool for proving our main result.
 If $a\in \mathbb{R}$, we write $[a]$ for the integer part of $a$.

\begin{lemma}\label{lemma 1}
Let $\alpha >1,$ $\alpha\in \mathbb{R}$.
Let $\lambda \vdash n$ be such that
$\delta=\delta(\lambda)\ge 9\alpha$ and $\lambda_1,\lambda'_1\le \frac n \alpha.$
If $\lambda_1>\lambda_2>\dots >\lambda_\delta>\delta$ and
$\lambda'_1>\lambda'_2 >\dots > \lambda'_{\tau +1} \ge \delta$, for some $0\le \tau\le \delta$, then
there exists $n_0$ such that for all $n\ge n_0$
$$
f^\lambda \ge \alpha^{n-(\delta^2+\alpha\rho)},
$$
where $\rho=\delta^2$ if $\alpha\in \mathbb{N}$ and
$\rho=\left[ \frac{\delta^2}{\alpha-[\alpha ]} \right]+1$ if $\alpha \not\in \mathbb{N}.$
\end{lemma}

\begin{proof}
We shall assign a number $N$ to each cell of $\lambda,$ $1\le N \le n$, and
we shall denote by $h_N$ the corresponding hook number.

We shall split the cells of $\lambda$ into four eventually empty disjoint
 sets $T_1,\ldots, T_4$.We start by defining cells of type 1.

Let $s_1$ be the number of corner cells of $\lambda=\lambda^{(1)}.$
By assumption $s_1\ge \delta \ge 9\alpha > 2\alpha.$ We enumerate the corner cells
from top to bottom with the numbers $1, 2, \dots , s_1$ and we assign to each cell color 1.

Next we consider the new diagram $\lambda^{(2)}$ obtained by deleting the $s_1$ corner cells from $\lambda.$
Let $s_2$ be the number of corner cells of $\lambda^{(2)}.$
Since $s_2\ge \delta -1 > 2\alpha,$ we repeat the above procedure by enumerating the corner cells
of $\lambda^{(2)}$ from top to bottom with the numbers $s_1+1,  \dots , s_1+s_2$
and we assign color 2 to each cell.

We repeat this procedure until when the obtained uncolored subdiagram of $\lambda$
has $\ge 2\alpha$ corner cells. Let $1,2,\dots ,r$ be the assigned colors.
The cells obtained in this procedure will be called of type 1.
Hence $T_1= \lambda \setminus \lambda^{(r+1)}$ and $|T_1|=\sum_{i=1}^r s_i.$

We remark that if we consider a cell of type 1, then the cells in the corresponding hook
(whose corner is the given cell) are all colored and each color appears at most twice.
Hence if we consider a cell in the $i$-th step of the above procedure, the corresponding
hook number is $h_N\le 2(i-1)+1$ where $N$ is the number assigned to this given cell.
Hence, since $s_1\ge \delta$ and by hypothesis $s_j\ge 2\alpha,$ for all $j$, $2\le j\le r,$
by counting the number of cells up to the $(i-1)$-th step, we get
$N> 2\alpha(i-2)+ \delta.$ Since $\delta > \alpha$ it follows that
\begin{equation}\label{1}
\alpha h_N\le N.
\end{equation}

Next we claim that
\begin{equation}\label{2}
|T_1|\ge 2\alpha r+ \alpha\delta.
\end{equation}

In fact since the first $\delta$ rows of $\lambda$ have different lengths, we have
$$
s_1\ge \delta, \  s_2\ge \delta -1, \dots , s_i\ge \delta - i+1,\dots ,
$$
$$
s_{\delta-2[\alpha]-1}\ge 2[\alpha]+2>2\alpha, \
s_{\delta-2[\alpha]}\ge 2\alpha,\dots , s_r\ge 2\alpha.
$$

Since
$$
s_1\ge (2[\alpha]+2)+(\delta-2[\alpha]-2),\ s_2\ge (2[\alpha]+2)+(\delta-2[\alpha]-3),\dots ,
$$
$$
s_{\delta-2[\alpha]-2}\ge (2[\alpha]+2)+1, \
s_{\delta-2[\alpha]-1}\ge 2[\alpha]+2,
$$
we have
$$
|T_1|\ge 2\alpha r+ \sum_{j=1}^{\delta-2[\alpha]-2} j =
2\alpha r +\frac{(\delta-2[\alpha]-2)(\delta-2[\alpha]-1)}{2}.
$$

Recalling that $\delta \ge 9\alpha$ and $\alpha \ge [\alpha] \ge 1$,  we get
$$
\delta^2-4[\alpha]\delta+4[\alpha]^2
+6[\alpha] -3\delta +2-2\alpha\delta > 9\alpha\delta -4[\alpha]\delta-3\delta -2\alpha\delta\ge 0.
$$
Hence
$$
\frac{(\delta-2[\alpha]-2)(\delta-2[\alpha]-1)}{2}> \alpha\delta
$$
and the claim follows.

We continue the process of deleting cells from our diagram $\lambda$
and we define the set $T_2$ of cells of type 2.

Consider the diagram $\lambda^{(r+1)}$ and its corner cells $(i,j)$
where either $i>\delta$ or $j>\delta,$ i.e., we consider corner cells
outside the square diagram $\delta\times\delta .$

Let $t_1$ be the number of corner cells of  $\lambda^{(r+1)}$ outside the partition
$\left( \delta^\delta \right).$
If $t_1\ge \alpha,$ we enumerate these
cells with the numbers $|T_1|+1, \dots , |T_1|+t_1$ and we color them with color $r+1.$

Next we consider the diagram $\lambda^{(r+2)}.$ If the number $t_2$ of
corner cells outside $\left( \delta^\delta \right)$ is $\ge \alpha,$ we enumerate
these cells with the numbers $|T_1|+t_1+1, \dots , |T_1|+t_1+t_2$ and
we color them with color $r+2.$

We repeat this procedure as long as the number of corner cells
outside $\left( \delta^\delta \right)$ is $\ge \alpha$.
 Let $r+1, \dots , r+q$ be the
colors given to the cells of type 2. The number of cells of type 2
is $|T_2|=\sum_{i=1}^{q}t_i.$

Next we claim that the inequality (\ref{1}) holds for cells of type 2, i.e., for any $N$ with
$$
|T_1|+1\le N \le |T_1|+|T_2|.
$$

In fact, let $0\le x \le q-1$ and consider the corner cells of the diagram
$\lambda^{(r+x+1)}$ outside $\left( \delta^\delta \right).$ For any such cell $(i, j)$
whose number is $N$ we have that
$$
|T_1|< N \le |T_1|+t_1, \,\,\,\,\, \mbox{if}\,\,\,\, x=0
$$
and
$$
|T_1|+\sum_{k=1}^{x} t_k < N\le |T_1|+\sum_{k=1}^{x+1} t_k,\,\,\,\,\, \mbox{if}\,\,\,\, x>0.
$$

Hence
\begin{equation}\label{3}
N\ge |T_1|+\alpha x
\end{equation}
since $t_k\ge \alpha,$ $1\le k\le q.$

Next we compute $h_N.$ If $i\le \delta$ and $j>\delta,$ each cell of
the arm has different colors. Hence its length is $\lambda_i-j\le r+x.$
Since $j>\delta,$ the length of the leg is $\lambda'_j-i\le\delta-1.$
Therefore
\begin{equation}\label{4}
h_N\le r+x+\delta.
\end{equation}
By (\ref{2}) putting together (\ref{3}) and (\ref{4}) we get (\ref{1}).

In case $i>\delta$ and $j\le \delta,$ the length of the leg is $\le r+x$
and the length of the arm is $\le \delta-1$ and the result still holds.

Next we define the set $T_3$ of cells of type 3. We consider the partition
$\mu=\lambda^{(r+q+1)}.$

Notice that $\mu_\delta+\mu'_\delta< \alpha-1,$ if $\alpha\in \mathbb{N}$
and  $\mu_\delta+\mu'_\delta\le [\alpha]$ if $\alpha\not\in \mathbb{N}.$
In fact, otherwise $\mu$ will have $\ge \alpha$ corner cells
since the original partition has distinct lengths of rows and columns,
and, so,  we would be in type 2.

Cells of type 3 are defined as cells of
$\mu \setminus  \left( (\delta+\rho)^{\delta+\rho} \right).$
We enumerate cells of type 3 as follows. For $\delta+\rho< m\le \max \{\mu_1,\mu'_1 \}=k$
define
$$
Q_m=\{(i,j)\in \mu \mid i=m \ \mbox{or}\ j=m\}.
$$
First we enumerate the cells of $Q_k$
starting with $|T_1|+|T_2|+1$ in some order. Then we enumerated
 the cells of $Q_{k-1}$ in some order starting with $|T_1|+|T_2|+|Q_k|+1.$
 We continue this process until $Q_{\delta+\rho+1}.$

Consider a cell $(i, j)$ of type 3 whose number is $N.$ Then either
$(i,j)=(i, x+\delta),$ $1\le i\le \delta,$ or $(i,j)=(x+\delta ,j ),$ $1\le j\le \delta,$
where $x>\rho.$ From the enumeration of the cells of type 3, it follows that there are $\le \delta^2+cx$
cells whose number is larger than N, where $c=\alpha-1$ if  $\alpha\in \mathbb{N}$
and $c=[\alpha]$ if $\alpha\not\in \mathbb{N}.$ Hence $N\ge n-(\delta^2+cx).$
Also
$$
h_N\le \left(\frac n \alpha -x-\delta \right)+\delta = \frac n \alpha -x.
$$

If $\alpha\in \mathbb{N},$ then $c=\alpha-1$ and we have
$$
N\ge n-\delta^2-(\alpha-1)x = n-\alpha x +x-\delta^2 > \alpha h_N,
$$
since $x>\rho=\delta^2.$

If $\alpha\not\in \mathbb{N},$ then $c=[\alpha]$ and we have
$$
N\ge n-\delta^2-[\alpha]x = n-\alpha x +(\alpha-[\alpha])x-\delta^2 > \alpha h_N,
$$
since $x>\rho = \left[ \frac{\delta^2}{\alpha-[\alpha ]} \right]+1> \frac{\delta^2}{\alpha-[\alpha ]}.$

We have proved that the inequality (\ref{1}) holds for cells of type 3.

Finally we say that a cell of $\lambda$ is of type 4 if it is not of type 1,2 or 3.
Let $T_4$ be
the set of cells of type 4. So $|T_4|=n-|T_1|-|T_2|-|T_3|.$ Note that
\begin{equation}\label{6}
|T_4|\le \delta^2+\alpha\rho.
\end{equation}

We consider a standard tableau $T_\lambda=(t_{ij})$ of shape $\lambda$
such that $n-|T_4|+1\le t_{ij} \le n,$ for any cell $(i, j)$ of type 4. By Remark \ref{r1}
if $N=N_{ij},$ then
\begin{equation}\label{5}
h_{ij}=h_N\le N,
\end{equation}
for any cell $(i, j)$ of type 4.

We are now ready to compute a lower bound of $f^\lambda.$

We have
$$
f^\lambda= \frac {n!}{\prod h_{ij}}=
\frac {(n-|T_4|)!}{\prod_{(i, j)\in T_1\cup T_2\cup T_3}h_{ij}}\cdot \frac {n(n-1)\dots (n-|T_4|+1)}
{\prod_{(i, j)\in T_4}h_{ij}}=
$$
$$
= \prod_{N=1}^{|T_1|+|T_2|+|T_3|}\frac{N}{h_N} \cdot \prod_{N=|T_1|+|T_2|+|T_3|+1}^{n}\frac{N}{h_N}\ge
\alpha^{n-|T_4|}\ge \alpha^{n - (\delta^2+\alpha\rho)},
$$
where we have applied the inequalities (\ref{1}) for cells of type 1, 2 and 3 and
(\ref{5}) for cells of type 4. The last inequality follows from (\ref{6}).

This completes the proof of Lemma \ref{lemma 1}.
\end{proof}

\begin{proposition}\label{prop 1}
Let $\lambda \vdash n$ and $\alpha\in \mathbb{R},$ $\alpha >1.$ Suppose that
$\lambda_1,\lambda'_1\le \frac n \alpha$ and $\delta=\delta(\lambda)\ge 18\alpha.$
Then
$$
f^\lambda \ge \alpha^{n-(\frac 5 2 \delta^2+\alpha\rho)},
$$
where $\rho=\delta^2$ if $\alpha\in \mathbb{N}$ and
$\rho=\left[ \frac{\delta^2}{\alpha-[\alpha ]} \right]+1$ if $\alpha \not\in \mathbb{N}.$
\end{proposition}
\begin{proof}
We may clearly assume that $n> \delta^2.$ We shall modify $\lambda$
so that we can apply Lemma \ref{lemma 1}.

Define $\widetilde \lambda$ as follows:
\begin{enumerate}
\item[1)]
for $1\le i \le \delta$ such that $\lambda_i\ge \delta+i-1,$
define $\widetilde\lambda_i=\lambda_i-(i-1);$ otherwise set $\widetilde \lambda_i=\delta;$

\item[2)]
for
$1\le j \le \delta,$ if $\lambda'_j\ge \delta+j-1,$ set
$\widetilde \lambda'_j=\lambda'_j-(j-1);$ otherwise set $\widetilde \lambda'_j=\delta.$
\end{enumerate}

Note that we erase at most $\delta(\delta-1)$ cells. So $\widetilde \lambda\vdash n_1\ge n-\delta^2+\delta.$

Let $s$ be the largest integer such that $\widetilde \lambda_s>\delta,$ otherwise set $s=0$
if $\widetilde \lambda_i\le \delta$, for all $i$, $1 \le i \le \delta$.
Let also $t$ be the largest integer such that
$\widetilde \lambda'_t>\delta,$ otherwise set $t=0$ if $\widetilde \lambda'_i\le \delta$,
 for all $i$, $1 \le i \le \delta$.

By eventually  considering the conjugate partition, we may assume that $s\ge t.$
Since $n>\delta^2,$ then $s\ge 1.$

If $s=\delta,$ then set $\mu=\widetilde\lambda$.
Otherwise we define a new partition
$\mu\vdash n_2=n_1-\frac{(\delta-s-1)(\delta-s)}{2}$ as follows:
\begin{enumerate}
\item[1)]
$\mu_i=\widetilde\lambda_i,$ if $1\le i \le s+1,$
or
$\delta+1\le i,$
\item[2)]
$\mu_{s+2}=\delta-1, \dots ,$
$\mu_\delta=\delta-(\delta-s-1).$
\end{enumerate}
Notice that the largest square inside $\mu$ is $\left( \delta(\mu)^{\delta(\mu)}\right)$
where $\delta(\mu)\ge \left[ \frac{\delta(\lambda)}2\right]+1\ge 9\alpha.$
Next we shall apply Lemma \ref{2} for the partition $\mu\vdash n_2$.
Let $\rho(\mu)=\delta(\mu)^2$  if $\alpha\in \mathbb{N}$ and
$\rho(\mu)=\left[ \frac{\delta(\mu)^2}{\alpha-[\alpha ]} \right]+1$ if $\alpha \not\in \mathbb{N}.$

Since $\rho(\mu)\le \rho = \rho(\lambda)$ and $s\le \delta$, by Lemma \ref{2} we have
$$
f^\mu \ge \alpha^{n_2 -(\delta(\mu)^2 +\alpha \rho(\mu))}
$$
$$
\ge \alpha^{n_2 -(\delta(\lambda)^2 +\alpha \rho(\lambda))}
= \alpha^{n - \delta^2+ \delta -\frac{(\delta -s-1)(\delta - s)}{2} -\delta^2- \alpha\rho}
$$
$$
\ge \alpha^{ n- \frac{5\delta^2-(2\delta s-s^2)-(3\delta-s)}{2} -\alpha \rho}
\ge \alpha^{n-(\frac 5 2 \delta^2+\alpha\rho)}.
$$
Since $f^\lambda \ge f^\mu$ the proof is complete.
\end{proof}

\begin{theorem}
Let $\alpha\in \mathbb{R},$ $\alpha >1$ and let $\{\lambda^{(n)}\}_{n\ge 1}$ be a sequence
of partitions, $\lambda^{(n)} \vdash n,$ such that $\lambda^{(n)}_1, {\lambda^{(n)}_1}'\le \frac n \alpha .$
Then for any $1<\beta < \alpha,$ there exists $n_0$ such that for all $n \ge n_0,$
$f^{\lambda^{(n)}}\ge \beta^n.$
\end{theorem}

\begin{proof}
In order to simplify the notation we write $\delta(\lambda^{(n)})=\delta(n).$
Let $\beta\in \mathbb{R}$ be such that $1<\beta < \alpha,$ and let
\begin{equation} \label{7}
0< \gamma \le \frac{\ln \alpha - \ln \beta}{\ln \alpha}.
\end{equation}
We partition $\mathbb{N}$ into three disjoint sets $\mathbb{N}=M_1\cup M_2\cup M_3,$
 where the $M_i$'s are defined as follows:
$$
M_1=\{ n\in \mathbb{N} \mid\delta(n) < 18\alpha\},
$$
$$
M_2=\{ n\in \mathbb{N} \mid \delta(n) \ge 18\alpha, \
\gamma n \le \frac{5}{2} \delta(n)^2 + \alpha \rho (n)\},
$$
$$
M_3=\{n\in \mathbb{N} \mid \delta(n) \ge 18\alpha,  \
\gamma n > \frac{5}{2} \delta(n)^2 + \alpha \rho (n)\}.
$$

By Proposition \ref{prop 3} there exists $n_1$ such that for all $n\ge n_1, n\in M_1,$
we have that $f^{\lambda^{(n)}}\ge \beta^n$.
In fact $f^{\lambda^{(n)}}\ge \frac {\alpha^n}{n^m}$
where $m=\frac{(3\delta(n)-1)\delta(n)}{2}< \frac{3\delta(n)^2}{2}< 486 \alpha^2.$

Suppose now that $n\in M_2.$ Then $\gamma n \le \frac{5}{2} \delta(n)^2 + \alpha \rho (n)$
and suppose first that $\alpha\in \mathbb{N}.$ Then in this case $\rho (n)=\delta(n)^2$
and $\gamma n \le \frac{5}{2} \delta(n)^2 + \alpha \rho (n)$ implies that
$\gamma n \le (\frac{5}{2} + \alpha )\delta(n)^2.$ Thus $\frac{\delta(n)^2}{n}\ge
\gamma(\frac{5}{2} + \alpha)^{-1} = \varepsilon > 0.$ By Proposition \ref{prop 2} there exists
$n_2$ such that for all $n\ge n_2$, $n\in M_2$  we have that
$f^{\lambda^{(n)}}\ge \beta^n.$

Suppose now that $\alpha\not\in \mathbb{N}.$ Then
$$
\gamma n \le \frac{5}{2} \delta(n)^2 + \alpha \left( \left[\frac{\delta(n)^2}{\alpha-[\alpha]}\right]+1 \right)
\le \frac{5}{2} \delta(n)^2 + \alpha \cdot \frac{\delta(n)^2}{\alpha-[\alpha]}+\alpha \le
$$
$$
\le \delta(n)^2\left(3 + \frac{\alpha}{\alpha-[\alpha]}\right),
$$
since $\frac12 \delta(n)^2> \alpha.$ Thus $\frac{\delta(n)^2}{n}\ge
\gamma(3 + \frac{\alpha}{\alpha-[\alpha]})^{-1} = \varepsilon > 0.$
By Proposition \ref{prop 2} as above we get
$f^{\lambda^{(n)}}\ge \beta^n$, for any $n\ge n_2, n\in M_2.$

By Proposition \ref{prop 1} there exists $n_3$ such that for all $n\ge n_3, n\in M_3,$
$$
f^{\lambda^{(n)}}\ge \alpha^{n-\gamma n}\ge \beta^n,
$$
since from (\ref{7}) we have that $(n-\gamma n)\ln \alpha \ge \ln \beta.$

If we now take $n_0=max \{ n_1, n_2, n_3\},$ we get that $f^{\lambda^{(n)}}\ge \beta^n$ for all $n\ge n_0.$
This completes the proof of the theorem.
\end{proof}

\end{document}